\documentclass[a4paper,12pt]{article}
\usepackage{latexsym,amsmath,amsthm,amssymb}
\usepackage{a4wide}
\usepackage[colorlinks]{hyperref}
\usepackage{marginnote}
\usepackage{color}
\usepackage[noblocks]{authblk}
\usepackage{enumerate}

\definecolor{darkergreen}{rgb}{.2,.7,.2}

\hypersetup{
  colorlinks,
  linkcolor=blue,
  citecolor=darkergreen,
  hidelinks,
  pdfinfo={
    Title={Fractional Sobolev and Hardy-Littlewood-Sobolev inequalities},
    Author={Gaspard Jankowiak, Van Hoang Nguyen},
    Keywords={26D10, 35K55, 46E35, 47G30, Fractional Sobolev inequality, Hardy-Littlewood-Sobolev inequality, best constant, stereographic projection, nonlinear diffusion, pseudodifferential operators}
  }
}

\theoremstyle{plain}
\newtheorem{theorem}{Theorem}
\newtheorem{proposition}[theorem]{Proposition}
\newtheorem{lemma}[theorem]{Lemma}
\newtheorem{corollary}[theorem]{Corollary}

\theoremstyle{definition}

\theoremstyle{remark}
\newtheorem{remark}{\bf Remark}

\renewcommand{\(}{\left(}
\renewcommand{\)}{\right)}

\newcommand{\ent}{\mathrm{Ent}} 

\newcommand{\HLS}{Hardy-Littlewood-Sobolev}
\newcommand{\G}{\mathcal G}
\newcommand{\F}{\mathcal F}
\newcommand{\J}{\mathsf J}
\newcommand{\Y}{\mathsf Y}

\def\R{\mathbb R}

\def\al{\alpha}
\def\om{\omega}
\def\be{\beta}
\def\ga{\gamma}
\def\le{\leq}
\def\ge{\geq}
\def\De{\Delta} 
\def\Gam{\Gamma}
\def\si{\sigma}
\def\lam{\lambda}
\def\ep{\epsilon}
\def\na{\nabla}
\def\pa{\partial}
\def\la{\left\langle} 
\def\ra{\right\rangle} 
\def\lt{\left}
\def\rt{\right}

\def\mS{\mathcal{S}}
\def\mF{\mathcal{F}}
\def\mG{\mathcal{G}}

\def\dWs{\dot{W}^s(\R^n)}

\numberwithin{equation}{section}

\title{Fractional Sobolev and Hardy-Littlewood-Sobolev inequalities}
\author[*]{Gaspard Jankowiak}
\affil[*]{CNRS -- \textsc{Ceremade, UMR7534} -- Universit\'e Paris-Dauphine}
\author[**]{Van Hoang Nguyen}
\affil[**]{School of Mathematical Sciences -- Tel Aviv University}

\begin{document}
\maketitle

\begin{abstract}
    This work focuses on an improved fractional Sobolev inequality with a remainder term involving the \HLS{} inequality which has been proved recently. By extending a recent result on the standard Laplacian to the fractional case, we offer a new, simpler proof and provide new estimates on the best constant involved. Using endpoint differentiation, we also obtain an improved version of a Moser-Trudinger-Onofri type inequality on the sphere. As an immediate consequence, we derive an improved version of the Onofri inequality on the Euclidean space using the stereographic projection.

\vfill
\noindent\rule{\linewidth}{0.4pt}
\par\noindent \textsuperscript{*}\texttt{jankowiak@ceremade.dauphine.fr}\hfill\textsuperscript{**}\texttt{vanhoang0610@yahoo.com}
\par\noindent ~\url{http://gaspard.janko.fr/}
\begin{center}
\textsuperscript{*}
Place de Lattre de Tassigny, 75016 Paris, \textsc{France}
\\
\textsuperscript{**}
School of Mathematical Sciences, Tel Aviv University, Tel Aviv 69978, \textsc{Israel}
\end{center}
\end{abstract}

\clearpage

\section{Introduction}
\label{sec:intro}

\par The sharp Sobolev inequality and the \HLS{} inequality are \emph{dual} inequalities. This has been brought to light first by Lieb~\cite{Lieb83} using the \emph{Legendre transform}. Later, Carlen, Carrillo, and Loss~\cite{CCL10} showed that the \HLS{} inequality can also be related to a particular Gagliardo-Nirenberg interpolation inequality via a fast diffusion equation. Since the sharp Sobolev inequality is in fact an endpoint in a familly of sharp Gargliardo-Nirenberg inequalities~\cite{DD02}, this eventually led to Dolbeault~\cite{Dol11} pointing out that a Yamabe type flow is related with the duality between the sharp Sobolev inequality, and the \HLS{} inequality. Still relying on that flow, he proved an enhanced Sobolev inequality, with a remainder term involving the \HLS{} inequality and also provided an estimate on the best multiplicative constant. This was soon extended to the setting of the fractional Laplacian operator by Jin and Xiong~\cite{JinXio11}. This approach heavily relies on the use of the fast diffusion equation, which introduces technical restrictions on the dimension or the exponent of the Laplacian operator. A simpler proof is provided in \cite{Dol-Jan14}, which lifts some of these restrictions, and provides better estimates on the best constant.

\par Let us now go into more details.  The sharp fractional Sobolev inequality states (see \emph{e.g.}~\cite{Swa92,CotTav04,Lieb83}) that
\begin{equation}\label{eq:fSI}
\lt(\int_{\R^n} |u(x)|^q dx\rt)^{\frac2q} \leq S_{n,s}\|u\|_s^2\qquad \text{for all } u\in \dot{W}^s(\R^n),
\end{equation}
where $0< s < \frac n2$, $q = \frac{2n}{n-2s}$, and the best constant $S_{n,s}$ is given by
\begin{equation}\label{eq:sharpSc}
    S_{n,s} = \frac{\Gam\left(\frac {n-2s}2\right)}{2^{2s}\pi^s\,\Gam\left(\frac{n+2s}2\right)}\lt(\frac{\Gam(n)}{\Gam\left( \frac n2 \right)}\rt)^{\frac{2s}n}.
\end{equation}
Moreover, equality in~\eqref{eq:fSI} holds if and only if $u(x) = c\, u_*\left(\frac{x-x_0}t\right)$ for some $c\in \R$, $t >0$, $x_0\in \R^n$ and where
\[
    u_*(x) = (1+ |x|^2)^{-(\frac n2-s)}
\]
is an Aubin-Talenti type extremal function.

The best constant $S_{n,s}$ has been computed first in the special cases $s=1$ and $n=3$ by Rosen \cite{Ros71}, and later for $s=1$ and $n\ge3$ by Aubin \cite{Aub76} and Talenti \cite{Tal76} independently. For general $0< s < \frac n2$, this best constant has been given by Lieb \cite{Lieb83} by computing the sharp constant in the sharp Hardy-Littlewood-Sobolev inequality,
\begin{equation}\label{eq:sHLSI}
    \lt|\iint_{\R^n\times\R^n} \frac{f(x)f(y)}{|x-y|^\lam}dxdy\rt| \leq \pi^{\frac \lam2}\frac{\Gam \left(\frac{n-\lam}2\right)}{\Gam \left(n-\frac\lam2\right)}\lt(\frac{\Gam(n)}{\Gam\left(\frac n2\right)}\rt)^{1-\frac\lam n} \|f\|_{L^p(\R^n)}^2\,,
\end{equation}
where $0 < \lam < n$ and $p = \frac{2n}{2n-\lam}$. There is equality in~\eqref{eq:sHLSI} if and only if $f(x) = c\, H_\lam(\frac{x-x_0}t)$ where
\[H_\lam(x) = (1+ |x|^2)^{-(n-\frac{\lam}2)},\]
with $c\in \R$, $t > 0$, and $x_0\in \R^n$.
For $0<s<\frac n2$, $2^{-2s}\pi^{-\frac n2}\frac{\Gamma( (n-2s)/2)}{\Gamma(s)}\frac{1}{|x|^{n-2s}}$
is the Green's function of $(-\De)^s$, so that the inequality~\eqref{eq:sHLSI} can be rewritten in the following equivalent form, by taking $\lambda = n-2s$
\begin{equation}\label{eq:sHLSIe}
    \lt|\int_{\R^n} f \, (-\De)^{-s}(f)\; dx\rt|\leq S_{n,s} \|f\|_{L^p(\R^n)}^2\,.
\end{equation}
The sharp Hardy-Littlewood-Sobolev inequality was first proved by Lieb based on a rearrangement argument (see \cite{Lieb83}). Recently, Frank and Lieb (see \cite{Frank-Lieb12}) have given a new and rearrangement-free proof of this inequality. Their method was also used to prove the sharp Hardy-Littlewood-Sobolev inequality in the Heisenberg group (see \cite{FL12}). See also \cite{CCL10,FL10} for the other rearrangement-free proofs for some special cases of the sharp Hardy-Littlewood-Sobolev inequality.

Using duality, Jin and Xiong state in \cite[Theorem 1.4]{JinXio11} that when $0 < s < 1$, $n\geq 2$, and $n > 4\,s$, there exists a constant $C_{n,s}$ such that the following inequality
\begin{multline}\label{eq:Jin-Xiong}
     S_{n,s}\|u^r\|_{L^{\frac{2n}{n+2s}}(\R^n)}^2 - \int_{\R^n} u^r \, (-\De)^{-s}u^r dx
     \\
     \leq C_{n,s}\|u\|_{L^{\frac{2d}{d-2s}}(\R^n)}^{\frac{8s}{n-2s}}\lt[S_{n,s} \|u\|_s^2 -\|u\|_{L^{\frac{2n}{n-2s}}(\R^n)}^2\rt],
\end{multline}
holds for any positive $u \in \dot W^s(\R^n)$, where~$r = \frac{n+2s}{n-2s}$. Moreover, the best value $C_{n,s}^*$ for the constant $C_{n,s}$ is such that $C_{n,s}^* \le \frac{n+2s}{n}\left( 1-e^{-\frac{n}{2s}} \right)S_{n,s}$. This adapts to the fractional setting the original result of Dolbeault \cite[Theorem 1.2]{Dol11} which was restricted to the case $s=1$.

In \eqref{eq:Jin-Xiong}, the left-hand side is positive by the \HLS{} inequality \eqref{eq:sHLSIe}, and the right-hand side is positive by Sobolev inequality \eqref{eq:fSI}, so this is an improvement of the Sobolev inequality.

The strong condition on the dimension required for \eqref{eq:Jin-Xiong} stems from the heavy reliance on a fast diffusion flow to achieve these results. Although the constraint on $n$ can be removed by lifting the flow to the sphere, Dolbeault and Jankowiak propose in \cite{Dol-Jan14} a new, simpler proof that brings a number of benefits in the case $s=1$: the role of duality is made more explicit, and it holds for any $n \ge 3$.

The aim of this paper is to extend and unify these results in the fractional setting.
We provide a better estimate on the best constant and by taking limits in $s$, we also derive an improved Moser-Trudinger-Onofri inequality, and recover the Onofri inequality for $n=2$. Our paper is organized as follows: in Section \ref{sec:results} we detail our results, both in the Sobolev (Theorem~\ref{maintheorem}) and Moser-Trudinger-Onofri (Theorem~\ref{improveMT}) settings. Sections \ref{sec:square} and \ref{sec:linearization} are dedicated to the proof of our main theorem using a completion of the square and linearization techniques, respectively.
Next we provide a proof of Theorem~\ref{improveMT} in Section \ref{sec:moser-trudinger}, by taking the limit $s\to \frac n2$.
Finally, in Section \ref{sec:flow}, we complete the proof of Theorem~\ref{maintheorem} using a fractional nonlinear diffusion flow.

\clearpage
\section{Results}
\label{sec:results}

Let us first introduce notation. First recall the definition of the homogeneous Sobolev space $\dot{W}^s(\R^n)$ with $s \in \R$. A Borel function $u\colon \R^n\to \R$ is said to vanish at the infinity if the Lebesgue measure of $\{x\in\R^n\, :\, |u(x)|>t\}$ is finite for all~$t>0$. For $s \in \R$, we define the fractional Laplace operator $(-\De)^s u$ by the distributional function whose Fourier transform is $|\xi|^{2s} \hat{u}(\xi)$, where $\hat{u}$ is the Fourier transform of $u$. For a test function $u$ in the Schwartz space $\mathcal{S}(\R^n)$, $\hat{u}$ is defined as $\hat{u}(\xi) = \int_{\R^n} e^{-i \la x,\xi\ra} u(x) dx$.
From the Plancherel-Parseval identity, we have $\|u\|_{L^2(\R^n)} = (2\pi)^{-\frac n2}\|\hat{u}\|_{L^2(\R^n)}$. We know that the Fourier transform is extended to a bijection from the space of the tempered distributions to itself. Then $\dot{W}^s(\R^n)$ is defined to be the space of all tempered distributions $u$ which vanishes at the infinity and~$(-\De)^{\frac s2}u \in L^2(\R^n)$. For $u\in \dot{W}^s(\R^n)$, we define
\[
    \|u\|_s^2 := \|(-\De)^{\frac s2} u\|_{L^2(\R^n)}^2 = \frac{1}{(2\pi)^n} \int_{\R^n} |\xi|^{2s} |\hat{u}(\xi)|^2\;d\xi
    = \int_{\R^n} u(x)\,(-\De)^s u(x) dx\,.
\]
With these notations, our main result is the following
\begin{theorem}\label{maintheorem}
Let $n\geq 2$, $0 < s < \frac n2$, and denote $r = \frac{n+2s}{n-2s}$
    \begin{enumerate}[(i)]
    \item There exists a positive constant $C_{n,s}$ for which the following inequality
        \begin{multline}\label{eq:mainresult}
            S_{n,s}\|u^r\|_{L^{\frac{2n}{n+2s}}(\R^n)}^2 - \int_{\R^n} u^r \,(-\De)^{-s}u^r dx \\
            \leq C_{n,s}\|u\|_{L^{\frac{2n}{n-2s}}(\R^n)}^{\frac{8s}{n-2s}}\lt(S_{n,s}\|u\|_s^2 - \|u\|_{L^{\frac{2n}{n-2s}}(\R^n)}^2\rt)
        \end{multline}
holds for any positive $u\in \dot{W}^s(\R^n)$.
    \item Let $C_{n,s}^*$ be the best constant in \eqref{eq:mainresult}. It is such that
        \begin{equation}\label{eq:bestconstant}
            \frac{n-2s+2}{n+2s+2} \, S_{n,s} \le C_{n,s}^*\le S_{n,s}\,.
        \end{equation}
        Additionally, in the case $0<s<1$ we know that:
        \begin{equation}\label{eq:bestconstant_upper}
            C_{n,s}^* < S_{n,s}\,.
        \end{equation}
    \end{enumerate}
\end{theorem}

Theorem~\ref{maintheorem} contains both the result of Dolbeault and Jankowiak \cite[Theorem 1]{Dol-Jan14} in the case $n\geq 3$ and $s = 1$ and the one of Jin and Xiong \cite[Theorem $4.1$]{JinXio11} in the case $s\in(0,1)$, $n \ge 2$ and $n > 4s$ for positive $u$.
The proof of Jin and Xiong is based on a fractional fast diffusion flow and some estimates on the extinction profiles. They also provide the upper bound $C_{n,s}^* \le \frac{n+2s}{n}(1 - e^{-\frac{n}{2s}})$, a bound which is larger that $1$ when $n > 4s$, so that Theorem~\ref{maintheorem} not only extends the result of Jin and Xiong to all $n\geq 2$ and $s\in (0,\frac{n}2)$, but also improve the constant $C_{n,s}$ on the right-hand side of~\eqref{eq:mainresult}.

Before continuing, we introduce the logarithmic derivative of the Euler Gamma function~$\Psi(a)= (\log \Gam(a))'$ for $a>0$,
and also define $\mathcal{H}_k$, the space spanned by $k$-homogeneous harmonic polynomials on $\R^{n+1}$ restricted to $S^n$.
In the following, $d\sigma$ denotes the normalized surface area measure on $S^n$ induced by the Lebesgue measure on $\mathbb R^{n+1}$.

In the spirit of \cite{Bec93,CL92}, we consider the limit $s \to \frac n2$ and obtain an inequality between the functionals associated with the \emph{Moser-Trudinger-Onofri} and the \emph{logarithmic} \HLS{} inequalities. Details will be given below, but let us first state our result.

\begin{theorem}\label{improveMT}
There exists a positive constant $C_n$ such that for any real-valued function $F$ defined on $S^n$ with an expansion on spherical harmonics $F =\sum_{k\geq 0} F_k$ where $F_k \in \mathcal{H}_k$, then the following inequality holds:
\begin{align}\label{eq:improveMT}
C_n\lt(\int_{S^n} e^F d\si\rt)^2 &\lt[\frac1{2n}\sum_{k\geq 1} \frac{\Gam(k+n)}{\Gam(n)\Gam(k)} \int_{S^n} |F_k|^2 d\si +\int_{S^n} F d\si - \log\lt(\int_{S^n} e^F d\si\rt)\rt]\notag\\
&\geq n\iint_{S^n \times S^n} e^{F(\xi)} \, \log |\xi -\eta| \, e^{F(\eta)} d\si(\xi) d\si(\eta)\notag\\
&\quad + \lt(\int_{S^n} e^F d\si\rt)^2\lt[\frac n2\lt(\Psi(n) -\Psi \left(\frac n2\right) -\log 4\rt) + \frac{\ent_\sigma(e^F)}{\int_{S^n} e^{F} d\si}\rt]\,,
\end{align}
where $\ent_\sigma(f) = \int_{S^n} f \log f\; d\sigma - (\int_{S^n} f\; d\sigma) \log (\int_{S^n} f\; d\sigma)$.
Moreover, if $C_n^*$ denotes the best constant for which the above inequality holds, then
\begin{equation}\label{eq:bestconstantMT}
\frac1{n+1} \leq C_n^* \leq 1.
\end{equation}
\end{theorem}

The inequality~\eqref{eq:bestconstantMT} is proved in the same way as item \emph{(ii)} in Theorem~\ref{maintheorem}. We will expand both sides of the inequality~\eqref{eq:improveMT} around the function $F\equiv 0$ which is an optimal function for the Moser-Trudinger-Onofri inequality~\eqref{eq:MTSphere}.

A direct consequence of Theorem~\ref{improveMT} written for $n=2$ is an improved version of the Euclidean Onofri inequality with a remainder term involving the two dimensional logarithmic Hardy-Littlewood-Sobolev inequality. We will use the following notation
\[d\mu(x) = \mu(x) dx,\qquad \mu(x) = \frac1{\pi\, (1+|x|^2)^2}, \qquad x\in \R^2.\]
\begin{corollary}\label{improveOnofri}
There exists a positive constant $C_2$ such that for any $f \in L^1(\mu)$ and $\na f \in L^2(\R^2)$, the following inequality holds:
\begin{align}\label{eq:improveOnofri}
C_2\lt(\int_{\R^2} e^f d\mu\rt)^2 &\lt[\frac1{16\pi} \|\na f\|_{L^2(\R^2}^2 + \int_{\R^2} f d\mu -\log\lt(\int_{\R^2} e^f d\mu\rt)\rt]\notag\\
&\geq \lt(\int_{\R^2} e^f d\mu\rt)^2\lt(1 + \log\pi + \int_{\R^2}\frac {e^f \mu}{\int_{\R^2} e^f d\mu} \log\lt(\frac{e^f \mu}{\int_{\R^2} e^fd\mu}\rt) dx \rt)\notag\\
&\quad -4\pi \int_{\R^2} e^{f(x)} \mu(x) \, (-\De)^{-1} (e^f \mu) (x)\, dx.
\end{align}
Moreover, if $C_2^*$ denotes the best constant for which the inequality~\eqref{eq:improveOnofri} holds, then
\[\frac13 \leq C_2^* \leq 1.\]
\end{corollary}
As above, the right-hand side of~\eqref{eq:improveOnofri} is nonnegative by the logarithmic Hardy-Littlewood-Sobolev inequality since Green's function of $-\De$ in $\R^2$ is given by $-\frac{1}{2\pi} \log\lt(|x|\rt)$. The inequality~\eqref{eq:improveOnofri} is a straightforward consequence of~\eqref{eq:improveMT} since $\Psi(2) -\Psi(1) =1$, and the fact that if $f(x) =F(\mS(x))$ with $\mS$ is the stereographic projection from $\R^2$ to $S^2$, then
\[\int_{\R^2} |\na f(x)|^2 dx = 4\pi \int_{S^2} |\na F|^2 d\si.\]
Another proof of Corollary~\ref{improveOnofri} is provided in Theorem $2$ of \cite{Dol-Jan14} by using a completely different method. More precisely, Dolbeault and Jankowiak use the square method to obtain an improved version of the Caffarelli-Kohn-Nirenberg inequalities on the weighted spaces, and then take a limit to get~\eqref{eq:improveOnofri}.

\bigskip

The proof of~\eqref{eq:mainresult} is similar to the one of Dolbeault and Jankowiak \cite{Dol-Jan14} which is based on the duality between the Sobolev and \HLS{} inequalities, in fact a simple expansion of a square integral functional. The first inequality in~\eqref{eq:bestconstant} is proved by expanding both sides of~\eqref{eq:mainresult} around the function $(1 + |x|^2)^{-\frac{n-2s}2}$ which is an extremal function for the fractional Sobolev inequality, and thus is a zero of both the left-hand side and right-hand side. To solve the linearized problem, we recast it to the unit sphere $S^n$ using the stereographic projection, and then identify the minimizers using the Funk-Hecke theorem (see \cite[Sec. $11.4$]{EMOT81}). The Funk-Hecke theorem gives a decomposition of $L^2(S^n)$ into the orthogonal summation of the spaces $\mathcal{H}_l$'s, that is
\begin{equation}\label{eq:decomp}
L^2(S^n) = \bigoplus_{l=0}^\infty\mathcal{H}_l,
\end{equation}
Moreover, the integral operators on $S^n$ whose kernels have the form $K(\la\om,\eta\ra)$ are diagonal with respect to this decomposition and their eigenvalues can be computed explicitly by using the Gegenbauer polynomials (see \cite[Chapter $22$]{AS72}).

By using stereographic projection, we can lift the sharp Hardy-Littlewood-Sobolev inequality~\eqref{eq:sHLSI} to the conformally equivalent setting of the sphere $S^n$ as follows
\begin{align}\label{eq:Sphereversion}
\left|\iint_{S^n\times S^n}  \frac{F(\xi) F(\eta)}{ |\xi-\eta|^{\lam}} d\si(\xi) d\si(\eta)\right| \leq B_\lam\lt(\int_{S^n} |F(\xi)|^p d\si(\xi)\rt)^{\frac 2p},
\end{align}
with
\[B_\lam = 2^{-\lam}\,\frac{\Gam\(\frac{n-\lam}2\)}{\Gam\(n-\frac\lam2\)}\,\frac{\Gam(n)}{\Gam\(\frac n2\)},\quad p =\frac{2n}{2n-\lam},\]
and $d\sigma$ is the normalized surface area measure on $S^n$. Note that the distance $|\cdot|$ is the distance in $\R^{n+1}$, not the geodesic distance on $S^n$. Some geometric and probabilistic informations can be obtained from this inequality through endpoint differentiation arguments (see \cite{Bec93}).
Carlen and Loss, but also Beckner considered the limit case of~\eqref{eq:Sphereversion} when $\lam =0$ while studying the two dimensional limit of the Sobolev interpolation inequality on the sphere, pioneered by Bidaut-V\'eron and V\'eron in \cite[Corollary 6.2]{BidVer91}.
In this limit, they proved the following Moser-Trudinger-Onofri inequality.
For any real valued function $F$ defined on $S^n$ with an expansion $F = \sum_{k\geq 0} F_k$, where $F_k \in \mathcal{H}_k$, the following holds
\begin{equation}\label{eq:MTSphere}
\log\lt(\int_{S^n} e^{F(\xi)} d\si(\xi)\rt) \leq \int_{S^n} F(\xi) d\si(\xi) + \frac1{2n} \sum_{k\geq 1}\frac{\Gam(n+k)}{\Gam(n)\Gam(k)} \int_{S^n} |Y_k(\xi)|^2 d\si(\xi).
\end{equation}
Moreover, equality holds in~\eqref{eq:MTSphere} if and only if
\[F(\xi) = -n\log |1 -\la\xi,\zeta\ra| +C,\]
for some $|\zeta| < 1$ and $C\in \R$.

When $n=2$, the inequality~\eqref{eq:MTSphere} becomes the classical Onofri inequality on $S^2$ (see \cite{Moser71, Onofri82}). Under the stereographic projection, this inequality is equivalent to the following inequality
\begin{equation}\label{eq:MOeu}
\log\lt(\int_{\R^2} e^{g(x)} d\mu(x)\rt) - \int_{\R^2} g(x) d\mu(x) \leq \frac1{16\pi} \int_{\R^2} |\na g(x)|^2 dx
\end{equation}
for any $g\in L^1(\mu)$ and $\na g\in L^2(\R^2)$.

The Onofri inequality~\eqref{eq:MOeu} plays the role of Sobolev inequality in two dimensions, see for example \cite{DEJ14} for a thorough review and justification of this statement. This inequality has several extensions, for instance to higher dimensions, which are out of the scope of this paper.

Just like the dual of the fractional Sobolev inequality is the \HLS{} inequality, the Legendre dual of~\eqref{eq:MTSphere} is the logarithmic \HLS{} inequality, first written in \cite{CL92} and \cite{Bec93}. It states that for nonnegative function $F$ such that $\int_{S^n} F d\si=1$,
\begin{multline}\label{eq:entropy}
-n \iint_{S^n\times S^n} F(\xi)\log |\xi-\eta|\, F(\eta) d\si(\xi) d\si(\eta) \\
\leq \frac n2 \lt(\Psi(n) -\Psi \left(\frac n2\right) -\log 4\rt) + \int_{S^n} F \log F \; d\sigma\,,
\end{multline}
where we recall $\Psi(a)= (\log \Gam(a))'$.
We remark that the appearance of the logarithmic kernel $-2\log|\xi -\eta|$ is quite natural since it is Green's function on $S^2$. We can rewrite inequality~\eqref{eq:entropy} in two dimensions and on the Euclidean space, and get that for any nonnegative function $f\in L^1(\R^2)$ such that $\int_{\R^2}f(x) dx = 1$, with $f\log f$ and~$(1 + \log |x|^2) f$ in $L^1(\R^2)$, we have
\begin{equation}\label{eq:logHLS}
\int_{\R^2}f\log f dx + 2\iint_{\R^2\times \R^2} f(x) \log|x-y|\,f(y)\, dx\, dy + (1+\log \pi) \geq 0.
\end{equation}
This more common version of the logarithmic \HLS{} inequality is the Legendre dual of the Onofri inequality~\eqref{eq:MOeu}. It has already seen a number of applications, \emph{e.g.} in chemotaxis models~\cite{CC08}.

\bigskip
In this paper, we take a step towards unification of the results of~\cite{Dol11, Dol-Jan14, JinXio11}. However, a number of questions remain unanswered.
The restriction $0<s<1$ in~\eqref{eq:bestconstant_upper} comes from the representation of the fractional Laplace operator, is this purely technical? To extend this part of the result to Theorem~\ref{improveMT}, it would make sens to consider a fractional logarithmic diffusion flow. However, this raises difficulties which are already presented in \cite[Proposition 3.4]{Dol11}, so we cannot exclude the case $C_n^* = 1$ yet. Finally, the computation of the exact value of $C_{n,s}^*$ is still open and probably requires new tools.

\label{sec:proofs}

\section{Upper bound on the best constant via an expansion of the square}
\label{sec:square}

In this section, we give a proof of Theorem~\ref{maintheorem} by the completion of the square method.

\begin{proof}[Proof of Theorem~\ref{maintheorem}]
By a density argument, it suffices to prove the inequality~\eqref{eq:mainresult} for any positive smooth function $u$ which belongs to Schwartz space on $\R^n$. For such functions, integration by parts gives us
\[\int_{\R^n} |\na (-\De)^{-\frac{1+s}2}v|^2\; dx = \int_{\R^n} v (-\De)^{-s}v\; dx,\]
and, if $v = u^r$ with $r= \frac{n+2s}{n-2s}$,
\[\int_{\R^n} \na (-\De)^{\frac{s-1}2}u \,\na (-\De)^{-\frac{1+s}2}v\; dx= \int_{\R^n} u(x) v(x)\; dx = \int_{\R^n} u(x)^q\; dx,\]
where $q = \frac{2n}{n-2s}$. Using these equalities, we have
\begin{multline}
0 \le \int_{\R^n}\lt|S_{n,s}\|u\|_{L^q(\R^n)}^{\frac{4s}{n-2s}}\na (-\De)^{\frac{s-1}s}u - \na(-\De)^{-\frac{1+s}2}v\rt|^2\; dx \\
=S_{n,s}^2\|u\|_{L^q(\R^n)}^{\frac{8s}{n-2s}}\|u\|_s^2- 2S_{n,s}\|u\|_{L^q(\R^n)}^{\frac{4s}{n-2s}}\int_{\R^n} u(x)^q\; dx + \int_{\R^n}u^r (-\De)^{-s} u^r\; dx.
\label{eq:square}
\end{multline}
Further, since $q = pr$, we have $\Vert u \Vert_{L^q(\R^n)}^q = \Vert u \Vert_{L^{pr}(\R^n)}^q = \Vert u^r \Vert_{L^{p}(\R^n)}^{p}$.
This shows that
\[\|u\|_{L^q(\R^n)}^{\frac{4s}{n-2s}}\int_{\R^n} u(x)^q\; dx = \Vert u^r \Vert_{L^p(\R^n)}^{p\frac{q-2}{q}} \Vert u^r \Vert_{L^p(\R^n)}^p = \|u^r\|_{L^p(\R^n)}^2.\]
Since the left hand side of~\eqref{eq:square} is nonnegative, it implies
\[
S_{n,s}\|u^r\|_{L^p(\R^n)}^2 - \int_{\R^n}u^r (-\De)^{-s} u^r\; dx
\leq S_{n,s}\|u\|_{L^q(\R^n)}^{\frac{8s}{n-2s}}\lt(S_{n,s}\|u\|_s^2 - \|u\|_{L^q(\R^n)}^2 \rt).
\]
This is exactly~\eqref{eq:mainresult} with $C_{n,s} = S_{n,s}$.
\end{proof}

\section{Lower bound via linearization}
\label{sec:linearization}

Let us start this section by briefly recalling some facts about the stereographic projection from the Euclidean space $\R^n$ to the unit sphere $S^n$. Denote $N =(0,\cdots,0,1)\in \R^{n+1}$ the \emph{north pole} of $S^n$ and consider the map $\mS\colon\R^n\mapsto S^n\setminus \{N\}$ defined by
\[\mS(x) = \lt(\frac{2x}{1+|x|^2}, \frac{|x|^2 -1}{1+|x|^2}\rt)\,,\]
the Jacobian of $\mathcal S$ is then given by
\[J_\mS(x) = \lt(\frac2{1+|x|^2}\rt)^n.\]
If $F$ is an integrable function on $S^n$ then $F(\mS(x))J_\mS(x)\in L^1(\R^n)$ and
\[\int_{\R^n } F(\mS(x))J_\mS(x)\; dx = \int_{S^n} F(\om) d\om,\]
where $d\om$ is the unnormalized surface area measure on $S^n$ induced by the Lebesgue measure on $\R^n$.
The inverse of $\mS$ is given by $\mS^{-1}(\om) = \(\frac{\om_1}{1-\om_{n+1}},\cdots, \frac{\om_n}{1-\om_{n+1}}\)$ with Jacobian $J_{\mS^{-1}}(\om)= (1-\om_{n+1})^{-n}$, where $\om = (\om_1,\om_2,\cdots,\om_{n+1}) \in S^n\setminus \{N\}$. Given $f\in \dWs$ and $q = \frac{2n}{n-2s}$, we define the new function $F$ on $S^n$ by
\begin{equation}\label{eq:fFcau}
F(\om) = f(\mS^{-1}(\om))J_{\mS^{-1}}(\om)^{\frac1q}.
\end{equation}
Then we have
\begin{equation}\label{eq:squarecau}
\int_{\R^n} \frac{f(x)^2}{(1+|x|^2)^{2s}}\; dx = 2^{-2s} \int_{S^n} F(\om)^{2}\,d\om,
\end{equation}
and
\begin{multline}\label{eq:HLScau}
\iint_{\R^n\times\R^n} \frac{f(x)^2}{(1+|x|^2)^{2s}} \,|x-y|^{-n+2s} \, \frac{f(y)^2}{(1+|y|^2)^{2s}}\; dx\; dy\\
 = 2^{-4s}\iint_{S^n\times S^n}F(\om)\, |\om-\eta|^{-n+2s}\, F(\eta) \;d\om\; d\eta.
\end{multline}
Equality~\eqref{eq:HLScau} is derived from the fact that
\[|\mS(x) -\mS(y)|^2 =\frac2{1+|x|^2}\, |x-y|^2\, \frac2{1+|y|^2}.\]

Next, we prove inequality~\eqref{eq:bestconstant}. For this purpose, let us denote $\mF$ and $\mG$ the positive functionals associated with the Sobolev and \HLS{} inequalities, respectively:
\begin{gather*}
\mF[u] = S_{n,s}\|u\|_s^2 -\|u\|_{L^q(\R^n)}^2,\qquad u\in \dWs\,,
\\
\mG[v] = S_{n,s}\|v\|_{L^p(\R^n)}^2 - \int_{\R^n}v (-\De)^{-s}v\; dx, \qquad v\in L^p(\R^n)\,,
\end{gather*}
and recall that $\mF[u_*] = 0$ and $\mG[u_*^r] = 0$. The inequality of Theorem~\ref{maintheorem} thus reads
\[
  C_{n,s}\,\Vert u \Vert_{L^q(\R^n)}^{\frac{8s}{n-2s}}\mF[u] \ge \mG[u^r]\,,
\]
and we are interested in a lower bound for
\[
  C_{n,s}^* = \sup_{u\in \dot W^s}\frac{\mG[u^r]}{\Vert u \Vert_{L^q(\R^n)}^{\frac{8s}{n-2s}}\mF[u]}\,.
\]

Consider now $u = u_* + \epsilon f$ where $f$ is smooth and compactly supported such that
\begin{equation}
  \int_{\R^n} \frac{u_*(x)\, f(x)}{(1+|x|^2)^{2s}}\; dx =0\,.
  \label{eq:ortho_cond_lin}
\end{equation}
By using the fact that $u_*$ is a critical point of $\mF$ and as such solves
\begin{equation}\label{eq:E-Le}
  (-\De)^s u_*(x) = \frac{2^{2s}\Gam\(\frac{n+2s}2\)}{\Gam\(\frac{n-2s}2\)}\, u_*(x)^r
  = \frac{2^{2s}\Gam\(\frac{n+2s}2\)}{\Gam\(\frac{n-2s}2\)}\, \frac{u_*(x)}{(1+|x|^2)^{2s}}\,,
\end{equation}
we in fact have the following.
\begin{proposition} With the above notation and~$f$ satisfying~\eqref{eq:ortho_cond_lin},
\begin{equation}\label{eq:mFsecond}
  \frac{\mF[u_\ep]}{S_{n,s}} = \ep^2\lt(\|f\|_s^2 - \frac{2^{2s}\Gam\(\frac{n+2s+2}2\)}{\Gam\(\frac{n-2s+2}2\)} \int_{\R^n}\frac{f(x)^2}{(1+|x|^2)^{2s}}\; dx\rt) + o(\ep^2).
\end{equation}
\end{proposition}
\begin{proof}
By a direct computation, we have
\[\frac{d}{d\ep}\lt(\mF[u_\ep]\rt)_{\ep =0} =2S_{n,s}\int_{\R^n}f(-\De)^su_*\; dx-2\lt(\int_{\R^n} u_*^q\; dx\rt)^{\frac2q-1}\int_{\R^n}u_*^{q-1} f\; dx = 0,\]
here, we use the fact that $(-\De)^su_*$ and $u_*^{q-1}$ are proportional to $u_*(x) (1+|x|^2)^{-2s}$. Taking the second derivative of $\mF[u_\ep]$ at $\ep = 0$, we obtain
\begin{align*}
\frac{d^2}{d\ep^2}\lt(\mF[u_\ep]\rt)_{\ep =0}& = 2S_{n,s}\|f\|_s^2 - 2(q-1)\lt(\int_{\R^n} u_*^q\; dx\rt)^{\frac2q-1}\int_{\R^n} u_*^{q-2} f^2\; dx\\
&=2S_{n,s}\lt(\|f\|_s^2 - \frac{2^{2s}\Gam\(\frac{n+2s+2}2\)}{\Gam\(\frac{n-2s+2}2\)} \int_{\R^n}\frac{f(x)^2}{(1+|x|^2)^{2s}}\; dx\rt).
\end{align*}
Since $\mF[u_*] =0$, using Taylor's expansion, we get~\eqref{eq:mFsecond}.
\end{proof}

Let us denote
\[\mathrm F[f] = \|f\|_s^2 - \frac{2^{2s}\Gam\(\frac{n+2s+2}2\)}{\Gam\(\frac{n-2s+2}2\)} \int_{\R^n}\frac{f(x)^2}{(1+|x|^2)^{2s}}\; dx.\]

Now, we introduce the new functions
\[f_0(x) = u_*(x),\quad f_i(x) = \frac{2x_i}{1+|x|^2}u_*(x),\, i =1,\cdots ,n,\quad f_{n+1}(x) = \frac{|x|^2-1}{1+|x|^2} u_*(x).\]
We remark that
\[f_i(x) = -\frac{2}{n-2s}\,\pa_{x_i} u_*(x) \qquad i =1,\cdots, n,\]
and
\[f_{n+1}(x) =-\frac{2}{n-2s}\,\pa_\lam\lt(\lam^{-(s-\frac n2)}u_*(\lam x)\rt)_{\lam =1}.\]
Using these relations and~\eqref{eq:E-Le}, we get
\begin{lemma}\label{lemmafi}
The following assertions hold:
\begin{equation}\label{eq:f0}
  (-\De)^s f_0(x) = \frac{2^{2s}\Gam\(\frac{n+2s}2\)}{\Gam\(\frac{n-2s}2\)} \frac{f_0(x)}{(1+|x|^2)^{2s}},
\end{equation}
\begin{equation}\label{eq:fi}
  (-\De)^s f_i(x) = \frac{2^{2s}\Gam\(\frac{n+2s+2}2\)}{\Gam\(\frac{n-2s+2}2\)}\frac{f_i(x)}{(1+|x|^2)^{2s}},\quad i =1,\cdots, n+1.
\end{equation}
\end{lemma}
We also notice that
\[\int_{\R^n} \frac{f_i(x) f_j(x)}{(1+|x|^2)^{2s}}\; dx = 0, \quad i, j = 0,1,\cdots, n+1, \quad i\not=j.\]

Next, we consider the other functional $\mG$ associated with the Hardy-Littlewood-Sobolev inequality as defined above.
\begin{proposition}
  With the above notation and~$f$ satisfying~\eqref{eq:ortho_cond_lin}, we have
\begin{equation}\label{eq:mGsecond}
\mG[(u_* + \epsilon f)^r] = \ep^2\lt(\frac{n+2s}{n-2s}\rt)^2\mathrm G[f] +o(\ep^2),
\end{equation}
where
\begin{align*}
\mathrm G[f] = \frac{\Gam(\frac{n-2s+2}2)}{2^{2s}\Gam(\frac{n+2s+2}2)}\int_{\R^n} \frac{f(x)^2}{(1+|x|^2)^{2s}}\; dx -\int_{\R^n}\frac{f(x)}{(1+|x|^2)^{2s}}(-\De)^{-s}\lt(\frac{f(x)}{(1+|x|^2)^{2s}}\rt)\; dx.
\end{align*}
\end{proposition}
\begin{proof}
First, $u_*^r$ solves the following integral equation which is the Euler-Lagrange equation associated with $\mG$:
\begin{equation}
  (-\De)^{-s} u_*^r =\frac{\Gam\(\frac{n-2s}2\)}{2^{2s}\Gam\(\frac{n+2s}2\)}\, u_*.
\end{equation}
Then
\begin{multline*}
  \frac{d}{d\ep}\lt(\mG[(u_* + \epsilon f)^r]\rt)_{\ep =0}
  = \frac{2S_{n,s}q}p \lt(\int_{\R^n}u_*^q\; dx\rt)^{\frac2p-1}\int_{\R^n}u_*^{q-1} f\; dx
  \\
  -2r\int_{\R^n} u_*^{r-1} f (-\De)^{-s}u_*^r\; dx =0,
\end{multline*}
since $u_*^{q-1}$ and $u_*^{r-1}(-\De)^{-s}u_*^r$ are proportional to $u_*(x) (1+|x|^2)^{-2s}$. By taking the second derivative, we get
\begin{align*}
\frac{d^2}{d\ep^2}\lt(\mG[(u_* + \epsilon f)^r]\rt)_{\ep =0}&= \frac{2S_{n,s}q(q-1)}p \lt(\int_{\R^n}u_*^q\; dx\rt)^{\frac2p-1}\int_{\R^n}u_*^{q-2} f^2\; dx\\
&\quad - 2r(r-1)\int_{\R^n} u_*^{r-2} f^2(-\De)^{-s} u_*^r\; dx\\
&\quad -2r^2 \int_{\R^n}u_*^{r-1}f (-\De)^{-s}(u_*^{r-1}f)\; dx\\
&= 2r^2\Bigg[\frac{\Gam\(\frac{n-2s+2}2\)}{2^{2s}\Gam\(\frac{n+2s+2}2\)}\int_{\R^n} \frac{f(x)^2}{(1+|x|^2)^{2s}}\; dx \\
&\quad-\int_{\R^n}\frac{f(x)}{(1+|x|^2)^{2s}}(-\De)^{-s}\lt(\frac{f(x)}{(1+|x|^2)^{2s}}\rt)\; dx\Bigg]\,.
\end{align*}
This concludes the proof.
\end{proof}

Next, by Legendre duality, we have
\begin{lemma}\label{dualitypro}
Suppose that $g$ satisfies the following conditions:
\begin{equation}\label{eq:ortho}
\int_{\R^n} \frac{g(x)f_i(x)}{(1+|x|^2)^{2s}}\; dx = 0, \quad i =1,\cdots, n+1.
\end{equation}
Then
\[\frac12 \int_{\R^n}\frac{g(x)^2}{(1+|x|^2)^{2s}}\; dx = \sup_f\lt(\int_{\R^n}\frac{f(x)g(x)}{(1+|x|^2)^{2s}}\; dx - \frac12 \int_{\R^n}\frac{f(x)^2}{(1+|x|^2)^{2s}}\; dx\rt),\]
and
\begin{align*}
\frac12 \int_{\R^n}\frac{g(x)}{(1+|x|^2)^{2s}}(-\De)^{-s}\lt(\frac{g(x)}{(1+|x|^2)^{2s}}\rt)\; dx =\sup_f \lt(\int_{\R^n}\frac{f(x)g(x)}{(1+|x|^2)^{2s}}\; dx -\frac12\|f\|_s^2 \rt),
\end{align*}
where supremum is taken over the functions $f$ satisfying the conditions~\eqref{eq:ortho}.

\end{lemma}
\begin{proof}
The proof of this proposition is elementary and is completely similar with the one of the dual formulas in~\cite{Dol-Jan14}.
\end{proof}

Given $f \in \dWs$, we consider the function $F$ defined by~\eqref{eq:fFcau} and its decomposition on spherical harmonics
\begin{equation}\label{eq:Fdecomp}
F(\om) = \sum_{k=0}^\infty F_k(\om),
\end{equation}
where $F_k \in \mathcal{H}_k$. Using the Funk-Hecke theorem and the dual principle for $\|\cdot\|_s$, we obtain the following.
\begin{lemma}\label{dWsdual}
  With $f$ and $F$ taken as in \eqref{eq:fFcau}-\eqref{eq:Fdecomp}, we have
\begin{equation}\label{eq:dWsdual}
  \|f\|_s^2 = \sum_{k=0}^\infty \frac{\Gam\(\frac{2k+n+2s}2\)}{\Gam\(\frac{2k+n-2s}2\)} \int_{S^n} F_k(\om)^2 d\om\,.
\end{equation}
\end{lemma}
\begin{proof}
We have
\begin{align*}
\|f\|_s^2& = \sup_{g} \lt(2\int_{\R^n} f(x)\, g(x)\; dx- \int_{\R^n} g(x)\, (-\De)^{-s} g(x)\; dx\rt)\\
&= \sup_{g}\lt(2\int_{\R^n} f(x)\, g(x)\; dx - \frac{\Gam\(\frac{n-2s}2\)}{\pi^{n/2}\,2^{2s}\,\Gam(s)} \iint_{\R^n\times\R^n} g(x) |x-y|^{-n+2s} g(y)\; dx\; dy\rt).
\end{align*}
Defining the function $G$ on $S^n$ by
\[G(\om) = g(\mS^{-1}(\om))J_{\mS^{-1}}(\om)^{\frac1p},\qquad p = \frac{2n}{n+2s},\]
and considering its decomposition $G =\sum_{k=0}^\infty G_k$, $G_k\in \mathcal{H}_k$, we then have
\begin{align*}
  2&\int_{\R^n} f(x)\, g(x)\; dx - \frac{\Gam\(\frac{n-2s}2\)}{\pi^{n/2}\,2^{2s}\,\Gam(s)} \iint_{\R^n\times\R^n} g(x) |x-y|^{-n+2s} g(y)\; dx\; dy\\
  &= 2\int_{S^n}F(\om)\,G(\om) d\om -\frac{\Gam\(\frac{n-2s}2\)}{\pi^{n/2}\,2^{2s}\,\Gam(s)} \iint_{S^n\times S^n} G(\om) |\om-\eta|^{-n+2s} G(\eta)\;d\om \;d\eta.
\end{align*}
Since $|\om-\eta|^{-n+2s} = 2^{-\frac{n-2s}2}(1-\la\om,\eta\ra)^{-\frac{n-2s}2}$, by \cite[Propostion $5.2$]{FL12} the integral operator with kernel $\frac{\Gam\(\frac{n-2s}2\)}{\pi^{n/2}\,2^{2s}\,\Gam(s)}|\om -\eta|^{-n+2s}$ is diagonal with respect to the decomposition~\eqref{eq:decomp}, and its eigenvalues are given by (see~\cite[Corollary $5.3$]{FL12})
\begin{equation}\label{eq:eigenvalue}
  \ga_k = \frac{\Gam\(\frac{2k+n-2s}2\)}{\Gam\(\frac{2k+n+2s}2\)}, \qquad k = 0, 1, 2\cdots.
\end{equation}
This implies that
\begin{align*}
  2\int_{\R^n} f(x)\, g(x)\; dx - &\frac{\Gam\(\frac{n-2s}2\)}{\pi^{n/2}\,2^{2s}\,\Gam(s)} \iint_{\R^n\times\R^n} g(x) |x-y|^{-n+2s} g(y)\; dx\; dy\\
&= \sum_{k=0}^{\infty}\lt(2\int_{S^n} F_k(\om)\,G_k(\om) d\om -\ga_k \int_{S^n} G_k(\om)^2 d\om\rt)\\
&\leq \sum_{k=0}^\infty \frac1{\ga_k} \int_{S^n} F_k(\om)^2 d\om.
\end{align*}
\end{proof}
As a consequence, if $f$ satisfies the conditions~\eqref{eq:ortho}, then $f$ satisfies the following Poincar\'e type inequality:
\begin{align}\label{eq:Poincaretype}
  \|f\|_s^2 \geq \frac{2^{2s}\Gam\(\frac{n+2s +4}2\)}{\Gam\(\frac{n-2s +4}2\)} \int_{\R^n} \frac{f(x)^2}{(1+|x|^2)^{2s}}\; dx.
\end{align}
Indeed, using the stereographic projection, we have
\[\int_{S^n} F(\om) d\om = \int_{\R^n} f(x)f_0(x) (1+|x|^2)^{-2s}\; dx =0,\]
and
\[\int_{S^n}F(\om)\, \om_i\; d\om = \int_{\R^n} f(x)\, f_i(x)\, (1+|x|^2)^{-2s}\; dx = 0,\qquad i =1,2,\cdots, n+1.\]
This shows that $F_0 = F_1 =0$ in the decomposition~\eqref{eq:Fdecomp} of $F$, then
\begin{align*}
  \|f\|_s^2 &\geq \frac{\Gam\(\frac{n+2s +4}2\)}{\Gam\(\frac{n-2s +4}2\)} \int_{S^n} F(\om)^2 d\om\\
  &= \frac{2^{2s}\Gam\(\frac{n+2s +4}2\)}{\Gam\(\frac{n-2s +4}2\)} \int_{\R^n} \frac{f(x)^2}{(1+|x|^2)^{2s}}\; dx.
\end{align*}

To sum up, we have
\begin{proposition}
\item (i) If $f\in \dWs$ satisfies the conditions~\eqref{eq:ortho} then
\begin{equation}\label{eq:Ff}
  \mathrm F[f] \geq \frac{4s}{n-2s+2}\, \frac{2^{2s}\Gam\(\frac{n+2s+2}2\)}{\Gam\(\frac{n-2s+2}2\)}\,\int_{\R^n}\frac{f(x)^2}{(1+|x|^2)^{2s}}\; dx.
\end{equation}
\item (ii) If $g$ satisfies the conditions~\eqref{eq:ortho} then
\begin{equation}\label{eq:Gg}
  \mathrm G[g] \geq \frac{4s}{n+2s+2}\, \frac{\Gam\(\frac{n-2s+2}2\)}{2^{2s}\Gam\(\frac{n+2s+2}2\)}\, \int_{\R^n} \frac{g(x)^2}{(1+|x|^2)^{2s}} \,\; dx.
\end{equation}
\end{proposition}
\begin{proof}
Item $(i)$ follows immediately from the definition of $\mathrm F[f]$ and \eqref{eq:Poincaretype},
while for~$(ii)$, from~\eqref{eq:Poincaretype} and Corollary~\ref{dualitypro}, we have
\[\int_{\R^n}\frac{g(x)}{(1+|x|^2)^{2s}}(-\De)^{-s}\lt(\frac{g(x)}{(1+|x|^2)^{2s}}\rt)\; dx \leq \frac{\Gam\(\frac{n-2s+4}2\)}{2^{2s}\Gam\(\frac{n+2s+4}2\)} \int_{\R^n} \frac{g(x)^2}{(1+|x|^2)^{2s}} \,\; dx.\]
Using the definition of $\mathrm G[g]$, we obtain~\eqref{eq:Gg}.
\end{proof}
\begin{corollary}
If $f\in \dWs$ and  satisfies the conditions~\eqref{eq:ortho}, then
\begin{equation}\label{eq:comFG}
  \mathrm G[f] \leq 2^{-4s}\,\frac{n-2s+2}{n+2s+2}\lt(\frac{\Gam\(\frac{n-2s+2}2\)}{\Gam\(\frac{n+2s+2}2\)}\rt)^2\, \mathrm F[f],
\end{equation}
and equality holds if and only if the function $F$ defined by~\eqref{eq:fFcau} belongs to $\mathcal{H}_2$.
\end{corollary}
\begin{proof}
Considering the function $F$ defined by~\eqref{eq:fFcau} and its decomposition $F = \sum_{k=2}^\infty F_k$, we know that
\[\|f\|_s^2 = \sum_{k\geq 2}\frac{1}{\ga_k}\int_{S^n} F_k(\om)^2 d\om,\]
where $\ga_k$ is given by~\eqref{eq:eigenvalue}.
Using equality~\eqref{eq:squarecau}, we also have
\[\int_{\R^n}\frac{f(x)^2}{(1+|x|^2)^{2s}}\; dx = 2^{-2s}\int_{S^n}F(\om)^2 d\om = 2^{-2s}\sum_{k=2}^\infty \int_{S^n}F_k(\om)^2 d\om.\]
From these equalities, we get
\begin{align}\label{eq:Ffsphere}
  \mathrm F[f]& = \sum_{k=2}^\infty\lt(\frac1{\ga_k} -\frac{\Gam\(\frac{n+2s+2}2\)}{\Gam\(\frac{n-2s+2}2\)}\rt) \int_{S^n}F_k(\om)^2 d\om\notag\\
&=\sum_{k=2}^\infty \al_k \int_{S^n}F_k(\om)^2 d\om,
\end{align}
with
\[\al_k = \frac{\Gam\(\frac{n+2s+2k}2\)\Gam\(\frac{n-2s+2}2\) -\Gam\(\frac{n-2s+2k}2\)\Gam\(\frac{n+2s+2}2\)}{\Gam\(\frac{n-2s+2k}2\)\Gam\(\frac{n-2s+2}2\)}.\]
Denote $g(x) = f(x) (1+|x|^2)^{-2s}$. Using the integral expression of $(-\De)^{-s}$ and equality~\eqref{eq:HLScau},
\begin{align*}
  \int_{\R^n} g(x)(-\De)^{-s}g(x)\; dx &= 2^{-4s} \frac{\Gam\(\frac{n-2s}2\)}{\pi^{n/2}\,2^{2s}\,\Gam(s)}\iint_{S^n\times S^n} F(\om)|\om-\eta|^{-n+2s} F(\eta) \;d\om \;d\eta\\
&=2^{-4s} \sum_{k\geq 2}\ga_k \int_{S^n} F_k(\om)^2 d\om.
\end{align*}
Therefore, we get
\begin{align}\label{eq:Ggsphere}
  \mathrm G[f]& =\sum_{k=1}^\infty\lt(\frac{\Gam\(\frac{n-2s+2}2\)}{2^{4s}\Gam\(\frac{n+2s+2}2\)}-\frac{\ga_k}{2^{4s}}\rt)\int_{S^n} F_k(\om)^2 d\om\notag\\
&= \frac1{2^{4s}} \sum_{k=2}^\infty \be_k \int_{S^n} F_k(\om)^2 d\om,
\end{align}
with
\[\be_k = \frac{\Gam\(\frac{n+2s+2k}2\)\Gam\(\frac{n-2s+2}2\) -\Gam\(\frac{n-2s+2k}2\)\Gam\(\frac{n+2s+2}2\)}{\Gam\(\frac{n+2s+2k}2\)\Gam\(\frac{n+2s+2}2\)}.\]
We have $\al_k, \be_k > 0$ for all $k\geq 2$. Moreover, we can prove that
\[\frac{\be_k}{\al_k} \leq \frac{\be_2}{\al_2} = \frac{n-2s+2}{n+2s+2}\lt(\frac{\Gam\(\frac{n-2s+2}2\)}{\Gam\(\frac{n+2s+2}2\)}\rt)^2,\quad \text{for all } k\geq 2,\]
and equality holds if $k=2$. From this inequality, we have
\[\mathrm G[f] = \frac1{2^{4s}} \sum_{k=2}^\infty \be_k \int_{S^n} F_k(\om)^2 d\om \leq 2^{-4s}\,\frac{n-2s+2}{n+2s+2}\lt(\frac{\Gam\(\frac{n-2s+2}2\)}{\Gam\(\frac{n+2s+2}2\)}\rt)^2\, \mathrm F[f].\]
This proves the inequality~\eqref{eq:comFG}. Additionally, we see from the proof that equality in~\eqref{eq:comFG} occurs if and only if $\int_{S^n}F_k(\om)^2 d\si(\om) = 0$ for all $k\geq 3$, hence $F\in \mathcal{H}_2$.
\end{proof}
As a consequence, we have
\begin{align}\label{eq:finish}
  \sup_f \frac{G(f)}{F(f)} =2^{-4s}\,\frac{n-2s+2}{n+2s+2}\lt(\frac{\Gam\(\frac{n-2s+2}2\)}{\Gam\(\frac{n+2s+2}2\)}\rt)^2,
\end{align}
where supremum is taken over $f\in \dWs$, $f\not=0$, and $f$ satisfying the conditions~\eqref{eq:ortho}.

We can now prove the first inequality in \eqref{eq:bestconstant} of Theorem~\ref{maintheorem}.

\begin{proof}[Proof of \eqref{eq:bestconstant}]
For all $f\in \dWs$, $f\not=0$ and $f$ satisfies the conditions~\eqref{eq:ortho}, denote $u_\ep = u_* +\ep f$, then
\[C_{n,s}^* \|u_\ep\|_{L^{\frac{2n}{n-2s}}(\R^n)}^{\frac{8s}{n-2s}}\geq \frac{\mG[u_\ep^r]}{\mF[u_\ep]}.\]
Let $\epsilon\to 0^+$, we get
\[C_{n,s}^* \geq \frac1{\|u_*\|_{L^{\frac{2n}{n-2s}}(\R^n)}^{\frac{8s}{n-2s}}S_{n,s}}\lt(\frac{n+2s}{n-2s}\rt)^2 \frac{G(f)}{F(f)}\]
Taking supremum over $f\in \dWs$, $f\not=0$, and $f$ satisfying the conditions~\eqref{eq:ortho}, using~\eqref{eq:finish} and the fact that
\[\int_{\R^n}u_*(x)^{\frac{2n}{n-2s}}\; dx = \int_{\R^n} (1+|x|^2)^{-n}\; dx = \pi^{\frac n2} \,\frac{\Gam\(\frac n2\)}{\Gam(n)},\]
we get
\[C_{n,s}^* \geq \frac{n-2s+2}{n+2s+2} S_{n,s}\]
as desired.
\end{proof}

\section{Improved Moser-Trudinger-Onofri inequality via endpoint differentiation}
\label{sec:moser-trudinger}

This section is dedicated to the proof of Theorem~\ref{improveMT}. By an approximation argument, it suffices to prove the inequality~\eqref{eq:improveMT} for bounded functions. We first prove for functions $F$ such that $\int_{S^n} F(\xi) d\xi = 0$. We define a new function $u$ on $\R^n$ by
\begin{equation}\label{eq:ufunc}
  u(x) = \lt(1+ \frac{n-2s}{2n} F(\mS(x))\rt) J_\mS(x)^{-(s-\frac n2)}\,.
\end{equation}
Since $F$ is bounded, then $u$ is positive when $s$ is close enough to $\frac n2$. Considering the expansion of $F$ in terms of spherical harmonics $F = \sum_{k\geq 1} F_k$ with $F_k \in \mathcal{H}_k$, it follows from Lemma~\ref{dWsdual} that
\[\|u\|_s^2 = |S^n|\frac{\Gam(\frac{n+2s}2)}{\Gam(\frac{n-2s}2)} + |S^n|\frac{(n-2s)^2}{4n^2}\sum_{k\geq 1} \frac{\Gam(\frac{2k+n+2s}2)}{\Gam(\frac{2k+n-2s}2)} \int_{S^n} F_k^2 d\si.\]
Using the stereographic projection, we get
\[\|u\|_{L^{\frac{2n}{n-2s}}(\R^n)}^2 = |S^n|^{\frac{n-2s}n}\lt(\int_{S^n} \lt(1 + \frac{n-2s}{2n} F\rt)^{\frac{2n}{n-2s}} d\si\rt)^{\frac{n-2s}n}.\]
For simplicity, we denote $t = \frac{n-2s}{2n}$, then
\begin{multline}\label{eq:xv}
S_{n,s}  \|u\|_{L^{\frac{2n}{n-2s}}(\R^n)}^{\frac{8s}{n-2s}}\lt( S_{n,s} \|u\|_s^2 - \|u\|_{L^{\frac{2n}{n-2s}}(\R^n)}^2\rt)\\
 = |S^n| \frac{\Gam(nt)}{\Gam(n(1-t))}\lt[\lt(\int_{S^n} \lt(1 + t F\rt)^{\frac 1t} d\si\rt)^{2-4t} - \lt(\int_{S^n} \lt(1 + t F\rt)^{\frac 1t} d\si\rt)^{2-2t}\rt]\\
\quad + |S^n| \frac{t^2 \Gam(nt)^2}{\Gam(n(1-t))^2} \lt[\sum_{k\geq 1}\frac{\Gam(k+n(1-t))}{\Gam(k+ nt)} \int_{S^n} F_k^2 d\si\rt]\lt(\int_{S^n} \lt(1 + t F\rt)^{\frac 1t} d\si\rt)^{2-4t}.
\end{multline}
Since $\Gam(nt) \sim 1/(nt)$ when $t\to 0^+$, by taking $t \to 0^+$ (or $s \to  \frac n2$) in~\eqref{eq:xv}, we obtain
\begin{multline}\label{eq:vp}
\lim\limits_{s \to \frac n2} \lt[S_{n,s}  \|u\|_{L^{\frac{2n}{n-2s}}(\R^n)}^{\frac{8s}{n-2s}}\lt( S_{n,s} \|u\|_s^2 - \|u\|_{L^{\frac{2n}{n-2s}}(\R^n)}^2\rt)\rt]\\
 = -\frac{2 |S^n|}{n \Gam(n)} \lt(\int_{S^n} e^F d\si\rt)^2 \log\lt(\int_{S^n} e^F d\si\rt)\\
\quad + \frac{|S^n|}{n^2 \Gam(n)^2}\lt[\sum_{k\geq 1}\frac{\Gam(k+n)}{\Gam(k)} \int_{S^n} F_k^2 d\si\rt] \lt(\int_{S^n} e^{F} d\si\rt)^2.
\end{multline}

We also have
\begin{equation}
\label{eq:vt1}
S_{n,s} \|u^{\frac{n+2s}{n-2s}}\|_{L^{\frac{2n}{n+2s}}(\R^n)}^2
= \frac{|S^n| \Gam(nt)}{\Gam(n(1-t))}\lt(\int_{S^n} (1 +tF)^{\frac{1}t}d\si\rt)^{2-2t}
\end{equation}

\begin{align}
\label{eq:vt2}
& \int_{\R^n} u^{\frac{n+2s}{n-2s}} (-\De)^{-s} u^{\frac{n+2s}{n-2s}}\; dx \\
&= \frac{|S^n|^2 \Gam(nt)}{4^s \pi^{\frac n2}\Gam(s)}\iint_{S^n\times S^n}
\frac{(1+tF(\xi))^{\frac{1-t}t}(1+t F(\eta))^{\frac{1-t}t}}{|\xi -\eta|^{2nt}}  \;d\si(\xi) \;d\si(\eta)\nonumber\\
&= \frac{|S^n|\Gam(n-nt)}{4^s \pi^{\frac n2}\Gam(s)}\lt(\int_{S^n} (1 +tF)^{\frac{1-t}t}d\si\rt)^{2}\nonumber\\
&\quad+ \frac{|S^n|^2 \Gam(nt)}{4^s \pi^{\frac n2}\Gam(s)}\iint_{S^n\times S^n}
\frac{(1+tF(\xi))^{\frac{1-t}t} (1+t F(\eta))^{\frac{1-t}t} }{\lt(|\xi -\eta|^{-2nt} -1\rt)^{-1}} \;d\si(\xi) \;d\si(\eta)\,.\nonumber
\end{align}
Letting $s\to \frac n2$ (\emph{i.e.} $t\to 0$) in~\eqref{eq:vt1}-\eqref{eq:vt2}, we obtain
\begin{multline}\label{eq:vt}
\lim\limits_{s\to \frac n2} \lt[S_{n,s} \|u^{\frac{n+2s}{n-2s}}\|_{L^{\frac{2n}{n+2s}}(\R^n)}^2  - \int_{\R^n} u^{\frac{n+2s}{n-2s}} (-\De)^{-s} u^{\frac{n+2s}{n-2s}}\; dx\rt]\\
=\frac{|S^n|}{\Gam(n)} \lt(\int_{S^n} e^F d\si\rt)^2\lt(\Psi(n) -\Psi\(\frac n2\) -\log 4 +\frac 2n \frac{\ent_\sigma(e^F)}{\int_{S^n} e^{F} d\si} \rt) \\
\quad + \frac{|S^n|}{\Gam(n)}\iint_{S^n \times S^n} e^{F(\xi)} \, \log \lt(|\xi -\eta|^2\rt) \, e^{F(\eta)} d\si(\xi) d\si(\eta)\,,
\end{multline}
where $\ent_\sigma(f) = \int_{S^n} f \log f\; d\sigma - (\int_{S^n} f\; d\sigma) \log (\int_{S^n} f\; d\sigma)$.

Now, applying the inequality~\eqref{eq:mainresult} to function $u$ defined by~\eqref{eq:ufunc}, then letting $s\to \frac n2$, and using the equalities~\eqref{eq:vp} and~\eqref{eq:vt}, we obtain
\begin{align}\label{eq:meanzero}
\lt(\int_{S^n} e^F d\si\rt)^2 &\lt[\frac1{2n}\sum_{k\geq 1} \frac{\Gam(k+n)}{\Gam(n)\Gam(k)} \int_{S^n} |F_k|^2 d\si - \log\lt(\int_{S^n} e^F d\si\rt)\rt]\notag\\
&\geq \frac n2\iint_{S^n \times S^n} e^{F(\xi)} \, \log \lt(|\xi -\eta|^2\rt) \, e^{F(\eta)} d\si(\xi) d\si(\eta)\notag\\
&\quad + \lt(\int_{S^n} e^F d\si\rt)^2\lt[\frac n2\lt(\Psi(n) -\Psi \(\frac n2\) -\log 4\rt) + \frac{\ent_\sigma(e^F)}{\int_{S^n} e^{F} d\si}\rt].
\end{align}
For any bounded function $F$, applying~\eqref{eq:meanzero} to function $F -\int_{S^n} F d\si$, we obtain~\eqref{eq:improveMT} with $C_n = 1$.

The above proof shows that $C_{n}^* \leq 1$. Let us now prove $C_n^* \geq \frac1{n+1}$. Indeed, for any function $F$ such that $\int_{S^n} F d\si =0$. Considering an expansion of $F$ by $F = \sum_{k\geq 1} F_k$, with $F_k \in \mathcal{H}_k$ and applying inequality~\eqref{eq:improveMT} to the function $\epsilon F$ with $\ep >0$, we get
\begin{multline}\label{eq:improveMTep}
C_n^*\lt(\int_{S^n} e^{\ep F} d\si\rt)^2 \lt[\frac{\ep^2}{2n}\sum_{k\geq 1} \frac{\Gam(k+n)}{\Gam(n)\Gam(k)} \int_{S^n} |F_k|^2 d\si - \log\lt(\int_{S^n} e^{\ep F} d\si\rt)\rt]\\
\geq \frac n2\iint_{S^n \times S^n} e^{\ep F(\xi)} \, \log \lt(|\xi -\eta|^2\rt) \, e^{\ep F(\eta)} d\si(\xi) d\si(\eta)\\
\quad + \lt(\int_{S^n} e^{\ep F} d\si\rt)^2\lt[\frac n2\lt(\Psi(n) -\Psi \(\frac n2\) -\log 4\rt) + \frac{\ent_\sigma(e^{\ep F})}{\int_{S^n} e^{\ep F} d\si}\rt].
\end{multline}
When $\ep$ is small, we have
\[\int_{S^n} e^{\ep F} d\si = 1 + \frac{\ep^2}2 \int_{S^n} |F|^2 d\si + o(\ep^2),\]
\[\ent_\sigma(e^{\ep F}) = \frac{\ep^2}2 \int_{S^n} |F|^2 d\si + o(\ep^2).\]
Moreover, since
\[\int_{S^n} \log\lt(|\xi -\eta|^2\rt) d\si(\eta) = -\lt(\Psi(n) -\Psi \(\frac n2\)-\log 4\rt)=:A(n),\]
then
\begin{multline*}
\iint_{S^n \times S^n} e^{\ep F(\xi)} \, \log \lt(|\xi -\eta|^2\rt) \, e^{\ep F(\eta)} d\si(\xi) d\si(\eta)
\\= A(n)+ \ep^2 A(n)\int_{S^n }|F|^2 d\si - \ep^2 \sum_{k\geq 1} \frac{\Gam(n)\Gam(k)}{\Gam(n+k)} \int_{S^n} |F_k|^2 d\si + o(\ep^2).
\end{multline*}
Substituting these above estimates into~\eqref{eq:improveMTep}, we obtain
\begin{multline*}
\frac{\ep^2}2 C_n^* \sum_{k\geq 2} \lt(\frac{\Gam(n+k)}{\Gam(n+1)\Gam(k)} -1\rt) \int_{S^n}|F_k|^2 d\si + o(\ep^2)\\
\geq \frac{\ep^2}2 \sum_{k\geq 2} \lt(1- \frac{\Gam(n+1)\Gam(k)}{\Gam(n+k)}\rt) \int_{S^n} |F_k|^2 d\si + o(\ep^2),
\end{multline*}
since $\Gam(n+1) = n\, \Gam(n)\, \Gam(1)$. If $F_k\not=0$ for some $k\geq 2$, then dividing both sides by $\frac{\ep^2}2$ and letting $\ep\to 0$, we get
\[C_n^*\geq \frac{\sum_{k\geq 2} \lt(1- \frac{\Gam(n+1)\Gam(k)}{\Gam(n+k)}\rt) \int_{S^n} |F_k|^2 d\si}{\sum_{k\geq 2} \lt(\frac{\Gam(n+k)}{\Gam(n+1)\Gam(k)} -1\rt) \int_{S^n}|F_k|^2 d\si}.\]
Taking supremum over $F = \sum_{k\geq 1} F_k$, $F_k\not=0$ for some $k\geq 2$, we obtain
\begin{align*}
C_n^* &\geq \sup\lt\{\frac{\sum_{k\geq 2} \lt(1- \frac{\Gam(n+1)\Gam(k)}{\Gam(n+k)}\rt) \int_{S^n} |F_k|^2 d\si}{\sum_{k\geq 2} \lt(\frac{\Gam(n+k)}{\Gam(n+1)\Gam(k)} -1\rt) \int_{S^n}|F_k|^2 d\si}\,: \, F =\sum_{k\geq 1} F_k,\, F_k\not=0\, \text{ for some } k\geq 2\rt\}\\
&= \frac1{n+1}.
\end{align*}
This completes the proof of Theorem~\ref{improveMT}.

\section{Fractional fast diffusion flow}
\label{sec:flow}

At this point, we know using the expansion of the square that $C_{n,s}^* \le S_{n,s}$, so that if we define
\[
  \mathcal C = \frac{C_{n,s}^*}{S_{n,s}}\, ,
\]
we know $\mathcal C \le 1$. In this section we will show that in fact $\mathcal C < 1$ when $0 < s < 1$. This condition is enforced throughout this section.
With the notations above, we consider the following fractional fast diffusion equation:
\begin{align}
  \partial_t v &+(-\Delta)^s v^m=0\,,\quad t > 0\,,\quad x\in \R^n\,,\quad m = \frac 1r  = \frac{n-2s}{n+2s}\,.
  \label{eq:frac-fde}
  \\
  v(0) &= v_0\,.\nonumber
\end{align}
which is well posed for $v_0 \in L^1 \bigcap L^\ell$ for some $\ell > \frac{2n}{n+2s}$ according to \cite[Theorem 2.3]{dePQuiRodVaz12}.
We will take initial datum $v$ with sufficient decay at infinity, \emph{e.g.} in the Schwartz space.

Let us define
\begin{align*}
  \G_0 = \G[v_0]\,
  \quad
  \J[v(t)] &= \int_{\R^n} v^p = \int_{\R^n} u^q\,,
  \quad
  \J_0 := \J[v_0]\,.
  \intertext{which is such that}
  \J' := \frac{d}{dt}\J &= -p \int_{\R^n} \left|(-\Delta)^{\frac{s}{2}} u\right|^2\,,
\end{align*}
We can now consider the evolution along the flow of the functional $\G$ associated
to the \HLS{} inequality. An easy computation gives
\[
  -\,\G'[v] = 2 \left( \int_{\R^n} v^{\frac{2n}{n+2s}} \right)^{\frac{2s}{n}}
  \F[v^m]
  = 2\, \J^{\frac{2s}{n}}\, \F[u]\,,
\]
which is nonnegative according to the fractional Sobolev inequality~\eqref{eq:fSI}. Hence,~$-\,\G[v]$ is nondecreasing and
stationary only when $u$ is an extremal function for~\eqref{eq:fSI}.
This and the following computations are a straightforward extension of those done in \cite{Dol11}.
Going one step further, we compute
\[
  -\,\G'' = -\frac{\J'}{\J} \,\G' - 4 \,m\, S_{n,s}\, \J^{\frac{2s}n} \mathsf K\,,
\]
with $\mathsf K = \int v^{m-1} \left| (-\Delta)^s v^m - \Lambda\, v\right|^2$, $\Lambda = \frac{n+2s}{2n} \frac{\J'}{\J}$. Then, using the fact that $\G' \le 0$, we have the following:

\begin{lemma}
  \label{lem:comparison}
  With the above notation and assuming $0 < s < 1$,
  \[
  \frac{\G''}{\G'} \le \frac{\J'}{\J}\,.
  \]
\end{lemma}

Using Lemma~\ref{lem:comparison} and \eqref{eq:fSI}, we have
\[
  -\,\G' \le \kappa_0\, \J \quad \text{ with } \quad \kappa_0 := \frac{-\,\G'(0)}{\J_0}\,
\]

Since $\J$ is nonincreasing in time, there exists $\Y : [0, \J_0] \to \R$ such that
\[
  \G(t) = \Y(\J(t))\,.
\]
Differentiating with respect to $t$ gives
\[
  -\Y'(\J)\, \J' = -\,\G' \le \kappa_0\, \J\,,
\]
then, substituting $\J'$ in the inequality of Theorem~\ref{maintheorem} (ii) we get
\[
  \mathcal C\left(-\frac{\kappa_0}{p}\, S_{n,s}^2 \frac{\J^{1+\frac{4s}{n}}}{\Y'} + S_{n,s}\, \J^{1+\frac{2s}{n}} \right) + \Y \le 0\,.
\]
With $\Y'=\frac d{dz}\Y$, we end up with the following differential inequality for $\Y$:
\begin{equation}
  \label{eq:ode_y}
  \Y' \left( \mathcal C\, S_{n,s}\, z^{1+\frac{2s}{n}} + Y \right)
  \le \mathcal C \frac{\kappa_0}{p}\, S_{n,s}^2\, z^{1+\frac{4s}{n}}\,,
  \quad
  \Y(0) = 0\,,
  \quad
  \Y(\J_0) = \G(0)\,.
\end{equation}

We have the following estimates. On the one hand
\[
  \Y' \le \frac{p}{\kappa_0}\, S_{n,s}\, z^{\frac {2s}n}
\]
and, hence,
\[
  \Y(z)\le\frac12\,\kappa_0\,S_{n,s}\,z^{1+\frac{2s}n}\quad\forall\,z\in[0,\J_0]\,.
\]
On the other hand, after integrating by parts on the interval $[0,\J_0]$, we get
\[
  \frac12\,\G(0)^2+\mathcal C\, S_{n,s}\,\J_0^{1+\frac{2s}n}\,\G(0)
  \le\frac14\,\mathcal C\,\kappa_0\,S_{n,s}^2\,\J_0^{2+\frac{4s}n}+\frac{n+2s}n\,\mathcal C\,S_{n,s}\,\int_0^{\mathsf J_0}z^\frac{2s}n\,\mathsf Y(z)\;dz\,.
\]
Using the above estimate, we find that
\[
  \frac2p\,S_{n,s}\,\int_0^{\mathsf J_0}z^\frac{2s}n\,\mathsf Y(z)\;dz\le\frac14\,\J_0^{2+\frac{4s}n}\,,
\]
and finally
\[
  \frac12\,\G_0^2-\mathcal C\,S_{n,s}\,\J_0^{1+\frac{2s}n}\,\G_0\le\frac12\,{\mathcal C}\kappa_0\,S_{n,s}^2\,\J_0^{2+\frac{4s}n}\,.
\]
Altogether, we have shown an improved inequality that can be stated as follows.

\begin{theorem}
  \label{thm:improved-nonlinear}
  Assume that $0<s<1$. Then we have
\begin{equation}
  0\le S_{n,s}\,\J^{1+\frac {2s}n}\,\varphi\left(\J^{\frac{2s}n-1}\, \F[u]\right)
  -\,\G[v]\,,
  \quad
\forall\,u\in\dot W^s(\R^n)\,,\;v=u^r
\end{equation}
where $\varphi(x):=\sqrt{\mathcal C^2+2\,\mathcal C\,x}-\mathcal C$ for any $x\ge0$.
\end{theorem}

\begin{proof} We have shown that for $u\in \mathcal S$, $y^2+\,2\,\mathcal C\,y-\,\mathcal C\,\kappa_0\le0$ with $y=\G_0/(S_{n,s}\,\J_0^{1+\frac{2s}n})\ge0$. This proves that $y\le\sqrt{\mathcal C^2+\mathcal C\kappa_0}-\mathcal C$, which proves that
\[
  \G_0\le S_{n,s} \,\J_0^{1+\frac{2s}n}\left(\sqrt{\mathcal C^2+\mathcal C\,\kappa_0}-\mathcal C\right)
\]
after recalling that
\[
  \frac12\,\kappa_0=-\frac{\G'_0}{\J_0}=\J^{\frac{2s}n-1}\,\F[u]\,.
\]
Arguing by density, we recover the results for $u\in \dot W^s(\R^n)$.
\end{proof}

\begin{remark}\label{def-phi}
  We may observe that $x\mapsto x-\varphi(x)$ is a convex nonnegative function which is equal to $0$ if and only if $x=0$. Moreover, we have
\[
\varphi(x)\le x\quad\forall\,x\ge0
\]
with equality if and only if $x=0$. However, one can notice that
\[
\varphi(x)\le\mathcal C\,x\quad\Longleftrightarrow\quad x\ge2\,\frac{1-\mathcal C}{\mathcal C}\,.
\]
\end{remark}

We recall that \eqref{eq:frac-fde} admits special solutions with separation of variables given by
\begin{equation}
  \label{eq:separation}
  v_*(t,x)=\lambda^{-(n+2s)/2}\,(T-t)^\frac{n+2s}{4s}\,u^{\frac{n+2s}{n-2s}}_*\left(\tfrac{x-x_0}{\lambda}\right)
\end{equation}
where $u_*(x):=(1+|x|^2)^{-\frac{n-2s}2}$ is an Aubin-Talenti type extremal function, $x\in\R^n$ and $0<t<T$. Such a solution is generic near the extinction time~$T$, see \cite[Theorem 1.3]{JinXio11}.

\begin{corollary}
  \label{cor:strict_upper_bound}
  With the above notations, $\mathcal C < 1$.
\end{corollary}

\begin{proof}
  Argue by contradiction and suppose $\mathcal C = 1$.
  Let $(u_k)$ be a minimizing sequence for the quotient $u \mapsto \frac{\F[u]}{\G[u^r]}$.
  Thanks to homogeneity, we can assume that $\J[u_k] = \J_* = \J[u_*]$ with $\J_*$ fixed, so that
  in fact $\G[u_k^r]$ is a bounded sequence.
  There are two possibilities.
  Either $\lim_{k\to\infty} \G[u_k^r] > 0$, and then, up to a subsequence, $\lim_{k\to\infty} \F[u_k] > 0$, and then
  \begin{multline*}
    0 = \lim_{k\to\infty} \left( S_{n,s}\, \J_*^{\frac{4s}n} \F[u_k] - \G[u_k] \right)
    \\
    = \lim_{k\to\infty} \left( S_{n,s}\, \J_*^{\frac{4s}n} \F[u_k] - S_{n,s}\, \J_*^{1+\frac{2s}n} \varphi\left( \J_*^{\frac{2s}n-1} \F[u_k] \right) \right)
    \\
    + \lim_{k\to\infty} \left( S_{n,s}\, \J_*^{1+\frac{2s}n} \varphi\left( \J_*^{\frac{2s}n-1} \F[u_k] \right) - \G[u_k] \right)\,.
  \end{multline*}
  The last term is nonnegative by Theorem~\ref{thm:improved-nonlinear}, and since $\lim_{k \to \infty} \F[u_k] > 0$, the first term is positive
  because of the properties of $\varphi$, see Remark~\ref{def-phi}.
  This is a contradiction, so in fact we have $\lim_{k \to \infty} G[u_k^r] = \lim_{k \to \infty} \F[u_k] = 0$.
  Since $\J[u_k] = \J_*$, $v_k = u_k^r$ maximizes
  \[
    \left\{ \int_{\R^n} v (-\Delta)^{-s} v \; dx \; : \; \Vert v \Vert_{\frac{2n}{n+2s}} = \J_*\right\}
  \]
  According to \cite[Theorem 3.1]{Lieb83}, up to translations and dilations, $v_k$ converges to $v_* = u_*^r$,
  and then the limit of the quotient $\frac{\mathcal F[u_k]}{\mathcal G[u_k^r]}$ is given by the linearization around
  the Aubin-Talenti profiles. That is
  \[
     \frac 1{S_{n,s}} = \lim_{k\to\infty} \frac{\mathcal F[u_k]}{\mathcal G[u_k^r]} \geq \frac{n+2s+2}{n-2s+2} \frac 1{S_{n,s}} \, ,
  \]
which is a contradiction. Thus, $C_{n,s}^* < S_{n,s}$.
\end{proof}

Inequality \eqref{eq:bestconstant_upper} holds by Corollary~\ref{cor:strict_upper_bound}, and the proof of Theorem~\ref{maintheorem} is complete.



\vfill
\hrule
\medskip
\noindent{\small{\bf Acknowlegments.} The authors want to thank Jean Dolbeault and Yannick Sire for their valuable input and comments. V.H.N is supported by a grant from the European Research Council and G.J. by the \emph{STAB}, \emph{NoNAP} and \emph{Kibord} (ANR-13-BS01-0004) projects of the French National Research Agency (ANR).
\par\medskip\noindent{\small\copyright\,2014 by the authors. This paper may be reproduced, in its entirety, for non-commercial purposes.}}

\clearpage
\bibliographystyle{abbrv}
\bibliography{JankowiakNguyen}

\end{document}